\documentclass[a4paper,12pt]{amsart}

\usepackage{amsmath,amssymb,amsthm,mathtools}
\usepackage{enumitem,needspace}
\usepackage{hyperref}
\usepackage[a4paper,textwidth=13.5cm,textheight=21cm,centering]{geometry}

\hypersetup{
  hidelinks,
  pdftitle={Closed Lindelof Sigma Kernels of Neighborhood Assignments in Charming Spaces},
  pdfauthor={Xing-Yu Hu},
  pdfsubject={General topology},
  pdfkeywords={charming space, D-space, neighborhood assignment, Lindelof Sigma-space, HFCw space, Hausdorff residue}
}

\newtheorem{theorem}{Theorem}[section]
\newtheorem{lemma}[theorem]{Lemma}
\newtheorem{proposition}[theorem]{Proposition}
\newtheorem{corollary}[theorem]{Corollary}
\theoremstyle{definition}
\newtheorem{definition}[theorem]{Definition}
\theoremstyle{remark}
\newtheorem{remark}[theorem]{Remark}

\numberwithin{equation}{section}

\newcommand{\Obs}{\mathsf{Obs}_{cc}}
\newcommand{\Core}{\mathsf{Core}_{cc}}
\newcommand{\doiref}[1]{doi:\nolinkurl{#1}}

\title[Closed assignment kernels in charming spaces]{Closed Lindel\"of $\Sigma$ Kernels of Neighborhood Assignments in Charming Spaces}
\author[X.-Y. Hu]{Xing-Yu Hu\textsuperscript{*}}
\thanks{\textsuperscript{*}Corresponding author: Xing-Yu Hu.}
\address{School of Mathematics and Statistics, Hanjiang Normal University, Shiyan, China}
\email{huxingyu@hjnu.edu.cn}
\date{}

\begin{document}

\begin{abstract}
Every neighborhood assignment on a charming space $X$ has a countable kernel whose closure in $X$ is Lindel\"of $\Sigma$. For a neighborhood assignment $P$ on $X$ and a set $A\subseteq X$, write $P[A]=\bigcup_{a\in A}P(a)$. Consequently, for any two neighborhood assignments $M$ and $N$ on $X$, there is a closed discrete set $D\subseteq X$ such that $M[N[D]]=X$. The $D$-property also admits an exact one-step reduction to covering a fixed Lindel\"of $\Sigma$ kernel by assigned neighborhoods centered at an ambient closed discrete set. The condition that every neighborhood assignment $\varphi$ on $X$ admit a pointwise square refinement, meaning a neighborhood assignment $V$ on $X$ with $V[V(x)]\subseteq\varphi(x)$ for every $x\in X$, is exactly the classical $z$-space condition formulated in terms of normal neighbornets. For a continuous closed surjection $p:Z\to X$, a fixed neighborhood assignment $\varphi$ on $X$ has a pointwise square refinement on $X$ if and only if there is a neighborhood assignment $V$ on $Z$ such that $p(V[V(z)])\subseteq\varphi(p(z))$ for every $z\in Z$. Every uncountable weakly hereditarily finally covering ($\mathrm{HFC}_{\mathrm w}$) subspace of $[\omega_1]^{\leq\omega}$ has a closed countable kernel for every neighborhood assignment but is not charming. For each Lindel\"of $\Sigma$ kernel $Y$ of $X$, put $B=\overline Y^{\,X}$ and $H=\overline{B\setminus Y}^{\,B}$. The Hu--Yin boundary obstruction is then $B\setminus Y$. The space $X$ has a closed Lindel\"of $\Sigma$ kernel exactly when there is a Lindel\"of $\Sigma$ kernel $Y$ such that, with $B=\overline Y^{\,X}$ and $H=\overline{B\setminus Y}^{\,B}$, the set $H\cap Y$ is $G_\delta$ in $H$. For any fixed Lindel\"of $\Sigma$ kernel $Y$, the closed-core reduction can be iterated. At every stage, the current iterate is a $D$-space if and only if $X$ is a $D$-space. The iteration reaches the largest closed subspace in which the kernel trace and its complement are both dense, and it stabilizes before the successor of the hereditary Lindel\"of number of $Y$. This bound is sharp. For each infinite cardinal $\kappa$, every $\gamma<\kappa^+$ occurs as the first stabilization stage for a compact space with a dense $\sigma$-compact Lindel\"of $\Sigma$ kernel whose cardinality and hereditary Lindel\"of number are both $\kappa$. The general charming-space $D$-problem remains open.
\end{abstract}

\maketitle

\section{Introduction}

A neighborhood assignment on a space $X$ is a function $\varphi$ that assigns to each $x\in X$ an open neighborhood $\varphi(x)$ of $x$.
The space $X$ is a $D$-space if every neighborhood assignment has a closed discrete kernel, that is, a closed discrete set $D\subseteq X$ for which
\[
X=\bigcup_{d\in D}\varphi(d).
\]
These definitions occur in van Douwen and Pfeffer \cite[p.~371]{vanDouwenPfeffer1979}. The classical $D$-space problem asks whether every regular Lindel\"of space must be a $D$-space. In a preprint, Tsaban constructs under CH a regular hereditarily Lindel\"of non-$D$ space of cardinality $\aleph_1$ and states that the corresponding existence assertion is independent of ZFC\@. It remains open whether it is consistent that every regular hereditarily Lindel\"of space is a $D$-space \cite[version~3, Theorem~1.1 and the discussion following it]{Tsaban2026}. Gruenhage \cite{GruenhageDsurvey} surveys the classical problem. The narrower question is whether every charming space is a $D$-space.

Arhangel'ski\u{\i} introduced charming spaces as a remainder-type extension of Lindel\"of $\Sigma$-spaces. A space $X$ is charming if it contains a Lindel\"of $\Sigma$ subspace $Y$ such that $X\setminus U$ is Lindel\"of $\Sigma$ whenever $U$ is an open neighborhood of $Y$ in $X$. Such a subspace $Y$ is called a Lindel\"of $\Sigma$ kernel of $X$ \cite[p.~93]{Arhangelskii2011}. Subsequent structural work on charming spaces and related strengthened subclasses appears in \cite{CasarrubiasPaniagua2016,LiLinLin2017,HuZhang2025}.

Hu and Yin \cite[Theorems~3.15 and~3.36]{HuYin2026} reduced the charming-space $D$-problem first to the dense-kernel case and then to a closed subspace defined from the boundary of a kernel. For a Lindel\"of $\Sigma$ kernel $Y$ of a charming space $X$, a set that is closed discrete in $Y$ need not be closed in $X$, since it may accumulate at points of $\overline Y^{\,X}\setminus Y$. The construction developed here gives, for each neighborhood assignment on $X$, a countable kernel whose closure in $X$ is Lindel\"of $\Sigma$.

For a neighborhood assignment $P$ on $X$ and $A\subseteq X$, write $P[A]=\bigcup_{a\in A}P(a)$. Theorem~\ref{thm:closed-LSigma-kernel} states that every neighborhood assignment on a charming space has a countable kernel whose ambient closure is Lindel\"of $\Sigma$. If $E\subseteq Y$ is nonempty and closed discrete in $Y$, then $E$ is countable because $Y$ is Lindel\"of. Lemma~\ref{lem:closure-cd-LSigma} then invokes \cite[Proposition~3.28]{HuYin2026}, which states for such an $E$ that $\overline E^{\,X}$ is charming with $E$ as a countable open dense Lindel\"of $\Sigma$ kernel and $\overline E^{\,X}\setminus E$ Lindel\"of $\Sigma$. Countable-union stability then shows that $\overline E^{\,X}$ is Lindel\"of $\Sigma$. The closure of the resulting countable kernel is a closed Lindel\"of $\Sigma$ assignment kernel. Its Lindel\"of $\Sigma$ property permits the $D$-property to be applied inside it, and its closedness ensures that the resulting closed discrete set remains closed discrete in $X$. Together, these two properties give Corollary~\ref{cor:two-assignments}. For any two neighborhood assignments $M$ and $N$ on $X$, there is a closed discrete set $D\subseteq X$ with $M[N[D]]=X$. Theorem~\ref{thm:exact-criterion} reduces the remaining one-step problem to covering a fixed Lindel\"of $\Sigma$ kernel by neighborhoods of an ambient closed discrete set. The condition that every neighborhood assignment $\varphi$ on $X$ admit a pointwise square refinement, meaning a neighborhood assignment $V$ on $X$ with $V[V(x)]\subseteq\varphi(x)$ for every $x\in X$, is classical. In neighbornet language it is precisely the $z$-space condition of Hicks and Sharma \cite[p.~109]{HicksSharma1979}. The elementary equivalence between one- and two-assignment refinements, together with the classical theory, shows that on an uncountable $\sigma$-compact metrizable space not every neighborhood assignment admits such a refinement. A related fiber-saturation construction appears in the proof of \cite[Proposition~2.13]{KunziFletcher1985}. Theorem~\ref{thm:closed-lift} gives the fixed-assignment equivalence for continuous closed surjections.

Van Mill, Tkachuk and Wilson introduced dual classes with respect to neighborhood assignments \cite[Definition~2.1]{vanMillTkachukWilson2007}. Related work includes weak dual classes \cite{AlasTkachukWilson2006,XuanSong2019}, results on kernels with prescribed topological properties, including dually Lindel\"of $\Sigma$ spaces \cite{BuzyakovaTkachukWilson2007}, the NA-$\mathcal P$/expansion-operator formulation of Basile, Bonanzinga, Maesano and Shakhmatov \cite{BasileBonanzingaMaesanoShakhmatov2023}, and the subsequent closure-operator treatment of neighborhood-assignment and star classes, in which ordinary topological closure is explicitly one of the allowed closure operators \cite[Definition~2.1]{MartinezRuizRamirezParamo2025}. Since every charming space is Lindel\"of, every neighborhood assignment on it already has a countable kernel \cite[Corollary~2.9]{vanMillTkachukWilson2007}, which is Lindel\"of $\Sigma$ by \cite[Proposition~1.2]{KubisOkunevSzeptycki2006}. The closure of any assignment kernel is again an assignment kernel. Thus the use of ordinary closure in neighborhood-assignment duality is not itself new. Theorem~\ref{thm:closed-LSigma-kernel} adds the simultaneous conclusion that one may choose a countable kernel whose ambient closure remains Lindel\"of $\Sigma$. Requiring every assignment to admit a kernel that is both closed and countable is the \emph{dually closed countable} property studied by Yu and Yun \cite[Section~4]{YuYun2009}. In particular, their Question~4.1 asks whether every Lindel\"of space has that property \cite[Question~4.1]{YuYun2009}. The theorem does not assert that the countable kernel itself is closed. It gives a countable kernel whose ambient closure is Lindel\"of $\Sigma$ and hence a closed separable assignment kernel.

For an uncountable weakly hereditarily finally covering ($\mathrm{HFC}_{\mathrm w}$) subspace $Z\subseteq[\omega_1]^{\leq\omega}$, Theorem~\ref{thm:hfcw-charming} shows that every Lindel\"of $\Sigma$ subspace is countable, the closure in $Z$ of every countable subset is countable, $Z$ is not charming, and $Z$ is dually closed countable. In particular, the $\diamondsuit$-construction in version~2 of Tsaban's preprint is a subspace of $[\omega_1]^{\leq\omega}$ and is $\mathrm{HFC}_{\mathrm w}$ \cite[version~2, Sections~3 and~4]{Tsaban2026}. Hence it is dually closed countable but not charming, independently of its announced non-$D$ property.

The boundary reduction in \cite[Theorem~3.36]{HuYin2026} admits an exact description. For a Lindel\"of subspace $Y$ of a Hausdorff space $X$ and $B=\overline Y^{\,X}$, every open neighborhood of a point of $B\setminus Y$ has a trace on $Y$ that is not countably compact. Thus the obstruction set defined in \cite[Definition~3.30]{HuYin2026} is precisely $B\setminus Y$, and the closed core in the earlier equivalence is $\overline{B\setminus Y}^{\,B}$. Theorem~\ref{thm:CL-equivalence} also shows that the kernel condition in \cite[Conjecture~3.42]{HuYin2026} is equivalent to the existence of a closed Lindel\"of $\Sigma$ kernel. The boundary reduction is then iterated while the original kernel remains fixed. On the traces arising from the iterated closed-core reduction, the induced derivative is precisely the transfinite iteration of Hausdorff's residue operation. Section~\ref{sec:boundary-iteration} proves this identity. Hausdorff's residue operation is classical \cite[p.~281]{Hausdorff1914}. Kuratowski treats residues of transfinite order in \cite[Section~12]{Kuratowski1966}. Here the residue iteration is coupled to the charming-space closed-core reduction and the $D$-property. Theorem~\ref{thm:core-iteration} propagates the resulting $D$-space equivalence through arbitrary limit stages and gives a stabilization bound by the successor of the hereditary Lindel\"of number of $Y$. Theorem~\ref{thm:uncountable-core-rank} shows that this bound is sharp. For each infinite cardinal $\kappa$, every $\gamma<\kappa^+$ is the first stabilization stage for a compact space with a dense $\sigma$-compact Lindel\"of $\Sigma$ kernel whose cardinality and hereditary Lindel\"of number are both $\kappa$.

Throughout the paper, all spaces are Tychonoff unless a weaker separation assumption is explicitly stated. Lindel\"of $\Sigma$-spaces are closed under closed subspaces and countable unions \cite[p.~2576]{KubisOkunevSzeptycki2006}, and every Lindel\"of $\Sigma$-space is a $D$-space \cite[p.~494]{Buzyakova2002}.

\section{Preliminaries}\label{sec:preliminaries}

For a space $X$, a neighborhood assignment $\varphi$ on $X$, and a set $A\subseteq X$, put
\[
\varphi[A]=\bigcup_{a\in A}\varphi(a).
\]
Thus, if $M$ and $N$ are neighborhood assignments on $X$ and $D\subseteq X$, then
\[
M[N[D]]=\bigcup\{M(x):x\in N[D]\}.
\]

\begin{definition}
For a neighborhood assignment $\varphi$ on a space $X$, a set $D\subseteq X$ is a \emph{kernel} of $\varphi$ if $\varphi[D]=X$ \cite[p.~689]{BuzyakovaTkachukWilson2007}. The space $X$ is a $D$-space if every neighborhood assignment on $X$ has a kernel that is closed discrete in $X$ \cite[p.~371]{vanDouwenPfeffer1979}.
\end{definition}

By \cite[Proposition~1.2]{KubisOkunevSzeptycki2006}, a space $Z$ is Lindel\"of $\Sigma$ if and only if it admits a cover $\mathcal C$ by compact subspaces and a countable family $\mathcal N$ of subsets of $Z$ such that, whenever $C\in\mathcal C$ and $U$ is an open neighborhood of $C$ in $Z$, some $N\in\mathcal N$ satisfies $C\subseteq N\subseteq U$. Such a family $\mathcal N$ is called a network with respect to $\mathcal C$. In particular, every countable space is Lindel\"of $\Sigma$, by taking $\mathcal C=\mathcal N=\{\{z\}:z\in Z\}$.

\begin{definition}
A space $X$ is \emph{charming} if there is a Lindel\"of $\Sigma$ subspace $Y\subseteq X$ such that, for every open set $U$ with $Y\subseteq U\subseteq X$, the subspace $X\setminus U$ is Lindel\"of $\Sigma$. Such a $Y$ is called a \emph{Lindel\"of $\Sigma$ kernel} of $X$ \cite[p.~93]{Arhangelskii2011}.
\end{definition}

Every charming space is Lindel\"of \cite[Proposition~3.3]{Arhangelskii2011}.

The following elementary closed-discrete union lemma is used below.

\begin{lemma}[{\cite[Lemma~2.14]{HuYin2026}}]\label{lem:finite-union-cd}
Let $E$ and $F$ be disjoint closed discrete subsets of a space $X$. Then $E\cup F$ is closed discrete in $X$.
\end{lemma}

\begin{lemma}\label{lem:closed-subspace-cd}
Let $X$ be a space and let $K\subseteq X$ be closed. If $D$ is closed discrete in $K$, then $D$ is closed discrete in $X$.
\end{lemma}

\begin{proof}
The set $D$ is closed in $K$, hence closed in $X$. For $d\in D$, an open set in $K$ isolating $d$ is the trace of an open set of $X$, which therefore isolates $d$ from $D$ in $X$.
\end{proof}

Closed heredity of the $D$-property is classical \cite[p.~11]{BorgesWehrly1991}.

\begin{lemma}[{\cite[Lemma~3.33]{HuYin2026}}]\label{lem:closed-hereditary-D}
Every closed subspace of a $D$-space is a $D$-space.
\end{lemma}

\section{The exact one-step assignment problem}

The following elementary criterion reduces the $D$-problem to covering a fixed Lindel\"of $\Sigma$ kernel by neighborhoods of an ambient closed discrete set.

\Needspace{9\baselineskip}
\begin{theorem}\label{thm:exact-criterion}
Let $X$ be a charming space and fix a Lindel\"of $\Sigma$ kernel $Y$ of $X$. The following are equivalent.
\begin{enumerate}[label=\textup{(\roman*)}]
\item $X$ is a $D$-space.
\item For every neighborhood assignment $\varphi$ on $X$, there exists a set $E\subseteq X$ that is closed discrete in $X$ and satisfies
\[
Y\subseteq\varphi[E].
\]
\end{enumerate}
\end{theorem}

\begin{proof}
Assume \textup{(i)} and let $\varphi$ be a neighborhood assignment on $X$. Choose a closed discrete kernel $E$ of $\varphi$. Then $\varphi[E]=X$, so in particular $Y\subseteq\varphi[E]$.

Conversely, assume \textup{(ii)} and let $\varphi$ be a neighborhood assignment on $X$. Choose a closed discrete $E\subseteq X$ with $Y\subseteq\varphi[E]$. Put
\[
R=X\setminus\varphi[E].
\]
The set $R$ is closed in $X$ and disjoint from $Y$. Since $\varphi[E]$ is an open neighborhood of the Lindel\"of $\Sigma$ kernel $Y$, the space $R$ is Lindel\"of $\Sigma$, and hence a $D$-space.

Restrict $\varphi$ to $R$ by
\[
\varphi_R(x)=\varphi(x)\cap R\qquad(x\in R).
\]
Choose a closed discrete $F\subseteq R$ such that $R=\varphi_R[F]\subseteq\varphi[F]$. By Lemma~\ref{lem:closed-subspace-cd}, $F$ is closed discrete in $X$. Also $E\cap R=\varnothing$, because every $e\in E$ belongs to $\varphi(e)\subseteq\varphi[E]$. Hence $E\cap F=\varnothing$. Lemma~\ref{lem:finite-union-cd} shows that $E\cup F$ is closed discrete in $X$, and
\[
X=\varphi[E]\cup R\subseteq\varphi[E\cup F].
\]
Thus $X$ is a $D$-space.
\end{proof}

\begin{remark}\label{rem:centers-outside}
The centers in Theorem~\ref{thm:exact-criterion}\textup{(ii)} are allowed to lie anywhere in $X$. Requiring $E\subseteq Y$ would impose an additional restriction and is not part of the equivalent formulation.
\end{remark}

\begin{corollary}\label{cor:compact-residual}
Let $X$ be a charming space and fix a Lindel\"of $\Sigma$ kernel $Y$ of $X$. The following are equivalent.
\begin{enumerate}[label=\textup{(\roman*)}]
\item $X$ is a $D$-space.
\item For every neighborhood assignment $\varphi$ on $X$, there is a closed discrete set $E\subseteq X$ such that $Y\setminus\varphi[E]$ is countably compact.
\item For every neighborhood assignment $\varphi$ on $X$, there is a closed discrete set $E\subseteq X$ such that $Y\setminus\varphi[E]$ is compact.
\end{enumerate}
\end{corollary}

\begin{proof}
Theorem~\ref{thm:exact-criterion} shows that \textup{(i)} implies both \textup{(ii)} and \textup{(iii)}, with empty residual. Conversely, assume \textup{(ii)} and let $\varphi$ be a neighborhood assignment on $X$. Let $E$ witness \textup{(ii)} for $\varphi$. The residual $C=Y\setminus\varphi[E]$ is closed in the Lindel\"of space $Y$, hence Lindel\"of. Since it is also countably compact, it is compact. By compactness, choose a finite set $F\subseteq C$ such that $C\subseteq\varphi[F]$. Since $E\subseteq\varphi[E]$ and $F\subseteq C=Y\setminus\varphi[E]$, the sets $E$ and $F$ are disjoint. The finite set $F$ is closed discrete in the Tychonoff space $X$. Lemma~\ref{lem:finite-union-cd} therefore shows that $E\cup F$ is closed discrete in $X$. Since $Y\subseteq\varphi[E\cup F]$, Theorem~\ref{thm:exact-criterion} gives \textup{(i)}. Thus \textup{(ii)} implies \textup{(i)}, and \textup{(iii)} implies \textup{(ii)}.
\end{proof}

\section{\texorpdfstring{Closed Lindel\"of $\Sigma$ assignment kernels}{Closed Lindelof Sigma assignment kernels}}

If $E\subseteq Y$ is nonempty, countable, and closed discrete in $Y$, Hu and Yin \cite[Proposition~3.28]{HuYin2026} show that $G=\overline E^{\,X}$ is a closed charming subspace with $E$ as a countable open dense Lindel\"of $\Sigma$ kernel and $G\setminus E$ Lindel\"of $\Sigma$. Since every closed discrete subspace of the Lindel\"of space $Y$ is countable, the following lemma combines \cite[Proposition~3.28]{HuYin2026} with the countable-union stability of Lindel\"of $\Sigma$-spaces \cite[p.~2576]{KubisOkunevSzeptycki2006}.

\Needspace{8\baselineskip}
\begin{lemma}\label{lem:closure-cd-LSigma}
Let $X$ be a charming space with Lindel\"of $\Sigma$ kernel $Y$. If $E\subseteq Y$ is closed discrete in $Y$, then $E$ is countable and
\[
G=\overline{E}^{\,X}
\]
is a closed Lindel\"of $\Sigma$ subspace of $X$.
\end{lemma}

\begin{proof}
Since $Y$ is Lindel\"of, the closed discrete set $E$ is countable. If $E=\varnothing$, then $G=\varnothing$, so the conclusion holds. Assume $E\ne\varnothing$. By \cite[Proposition~3.28]{HuYin2026}, the space $G$ is closed in $X$, the set $E$ is a Lindel\"of $\Sigma$ subspace of $G$, and $G\setminus E$ is Lindel\"of $\Sigma$. Countable-union stability therefore gives that $G=E\cup(G\setminus E)$ is Lindel\"of $\Sigma$.
\end{proof}

\begin{theorem}\label{thm:closed-LSigma-kernel}
Let $X$ be a charming space. Every neighborhood assignment $M$ on $X$ has a countable kernel $A$ such that
\[
K=\overline A^{\,X}
\]
is Lindel\"of $\Sigma$. Consequently, $K$ is a closed separable Lindel\"of $\Sigma$ subspace of $X$ and a kernel for $M$.
\end{theorem}

\begin{proof}
Fix a Lindel\"of $\Sigma$ kernel $Y$ of $X$. Since $Y$ is a $D$-space, the restricted neighborhood assignment
\[
y\longmapsto M(y)\cap Y
\]
has a closed discrete kernel $E\subseteq Y$. Thus
\[
Y\subseteq M[E].
\]
By Lemma~\ref{lem:closure-cd-LSigma},
\[
G=\overline E^{\,X}
\]
is closed in $X$ and Lindel\"of $\Sigma$.

Set
\[
R=X\setminus M[E]
\qquad\text{and}\qquad
L=G\cup R.
\]
The set $R$ is closed in $X$, disjoint from $Y$, and Lindel\"of $\Sigma$ because $M[E]$ is an open neighborhood of $Y$. Hence $L$ is closed in $X$ and Lindel\"of $\Sigma$. Since $E\subseteq L$ and every $r\in R$ belongs to $M(r)$,
\[
X=M[E]\cup R\subseteq M[L].
\]
Thus $L$ is a kernel of $M$.

Since every charming space is Lindel\"of, the open cover $\{M(\ell):\ell\in L\}$ of $X$ has a countable subcover indexed by some countable set $A\subseteq L$. Then $M[A]=X$, so $A$ is a countable kernel. Put $K=\overline A^{\,X}$. Because $L$ is closed in $X$, one has $K\subseteq L$. Hence $K$ is a closed subspace of the Lindel\"of $\Sigma$ space $L$ and is therefore Lindel\"of $\Sigma$. Since $A$ is dense in $K$, the space $K$ is separable. As $A\subseteq K$ and $M[A]=X$, one has $M[K]=X$, so $K$ is a kernel for $M$.
\end{proof}

Because every charming space is Lindel\"of, every neighborhood assignment on it has a countable kernel \cite[Corollary~2.9]{vanMillTkachukWilson2007}. The closure of any assignment kernel is again a kernel, so every assignment also has a closed kernel. Theorem~\ref{thm:closed-LSigma-kernel} combines these two conclusions by choosing a countable kernel whose ambient closure remains Lindel\"of $\Sigma$. This does not imply the \emph{dually closed countable} property of Yu and Yun \cite[Section~4]{YuYun2009}, which requires the countable kernel itself to be closed.

The closed Lindel\"of $\Sigma$ assignment kernel in Theorem~\ref{thm:closed-LSigma-kernel} allows a second application of the $D$-property. Because this kernel is closed in $X$, the closed discrete subset obtained inside it is also closed discrete in $X$.

\begin{corollary}\label{cor:two-assignments}
Let $X$ be a charming space and let $M$ and $N$ be neighborhood assignments on $X$. Then there exists a set $D\subseteq X$ that is closed discrete in $X$ and satisfies
\[
M[N[D]]=X.
\]
\end{corollary}

\begin{proof}
Apply Theorem~\ref{thm:closed-LSigma-kernel} to $M$ and choose a kernel $K$ for $M$ that is closed in $X$ and Lindel\"of $\Sigma$. Then
\[
M[K]=X.
\]
Since $K$ is Lindel\"of $\Sigma$, it is a $D$-space. Apply the $D$-property of $K$ to the neighborhood assignment
\[
x\longmapsto N(x)\cap K.
\]
Hence there is a closed discrete set $D\subseteq K$ such that
\[
K\subseteq N[D].
\]
Since $K$ is closed in $X$, Lemma~\ref{lem:closed-subspace-cd} shows that $D$ is closed discrete in $X$. Monotonicity of the set operation $A\mapsto M[A]$ gives
\[
X=M[K]\subseteq M[N[D]]\subseteq X.
\]
Thus $M[N[D]]=X$.
\end{proof}

\begin{corollary}\label{cor:two-step}
For every neighborhood assignment $\varphi$ on a charming space $X$, there exists a closed discrete set $D\subseteq X$ such that
\[
\varphi[\varphi[D]]=X.
\]
\end{corollary}

\begin{proof}
Take $M=N=\varphi$ in Corollary~\ref{cor:two-assignments}.
\end{proof}

Corollary~\ref{cor:two-step} differs from the $D$-property by exactly one iteration of the assignment operator. In general, this extra iteration cannot be removed by a pointwise square refinement.

\section{Square refinements and closed lifts}

\subsection{\texorpdfstring{The classical $z$-space condition}{The classical z-space condition}}

For comparison with the two-step conclusion of Corollary~\ref{cor:two-step}, recall the classical $z$-space condition. Hicks and Sharma define a $z$-space as a space in which every neighbornet is normal and give the equivalent covering formulation that every indexed open cover $\{G_x:x\in X\}$, with $x\in G_x$ for every $x\in X$, has an indexed open refinement $\{H_x:x\in X\}$ such that $y\in H_x$ implies $H_y\subseteq G_x$ \cite[p.~109]{HicksSharma1979}. In the language used here, this covering formulation says exactly that every neighborhood assignment $\varphi$ on $X$ admits a neighborhood assignment $V$ with $V[V(x)]\subseteq\varphi(x)$ for every $x\in X$. Junnila \cite{Junnila1978} develops the general theory of neighbornets. Proposition~\ref{prop:two-one-refinement} gives the elementary translation between this classical one-assignment formulation and the two-assignment form used below.

\begin{proposition}\label{prop:two-one-refinement}
Let $X$ be a space and let $\varphi$ be a neighborhood assignment on $X$. The following are equivalent.
\begin{enumerate}[label=\textup{(\roman*)}]
\item There are neighborhood assignments $M,N$ on $X$ such that
\[
M[N(x)]\subseteq\varphi(x)
\qquad(x\in X).
\]
\item There is a neighborhood assignment $V$ on $X$ such that
\[
V[V(x)]\subseteq\varphi(x)
\qquad(x\in X).
\]
\end{enumerate}
\end{proposition}

\begin{proof}
If $V$ witnesses \textup{(ii)}, then taking $M=N=V$ gives \textup{(i)}. Conversely, choose neighborhood assignments $M,N$ witnessing \textup{(i)} and define
\[
V(x)=M(x)\cap N(x).
\]
Then $V$ is a neighborhood assignment. If $y\in V(x)$ and $z\in V(y)$, then $y\in N(x)$ and $z\in M(y)$, so
\[
z\in M[N(x)]\subseteq\varphi(x).
\]
Hence $V[V(x)]\subseteq\varphi(x)$ for every $x$.
\end{proof}

Hicks and Sharma further note that every countable $T_1$ space is a $z$-space \cite[p.~110]{HicksSharma1979}, that the $z$-space property is hereditary \cite[p.~111]{HicksSharma1979}, and that every compact metrizable $z$-space is countable \cite[Corollary~3]{HicksSharma1979}. Combining these classical results with Proposition~\ref{prop:two-one-refinement} gives the following reformulation in the language of neighborhood assignments.

\begin{corollary}\label{cor:compact-metric-roots}
For a $\sigma$-compact metrizable space $X$, the following are equivalent.
\begin{enumerate}[label=\textup{(\roman*)}]
\item $X$ is countable.
\item Every neighborhood assignment $\varphi$ on $X$ admits neighborhood assignments $M,N$ on $X$ such that $M[N(x)]\subseteq\varphi(x)$ for every $x\in X$.
\item Every neighborhood assignment $\varphi$ on $X$ admits a neighborhood assignment $V$ on $X$ such that $V[V(x)]\subseteq\varphi(x)$ for every $x\in X$.
\end{enumerate}
\end{corollary}

\begin{proof}
By Proposition~\ref{prop:two-one-refinement}, conditions \textup{(ii)} and \textup{(iii)} are equivalent. If $X$ is countable, then $X$ is a $z$-space by \cite[p.~110]{HicksSharma1979}, so the covering characterization \cite[p.~109]{HicksSharma1979} gives \textup{(iii)}. Conversely, assume \textup{(iii)}. The covering characterization also shows that $X$ is a $z$-space. Write $X=\bigcup_{n\in\omega}K_n$ with each $K_n$ compact. Heredity of the $z$-space property \cite[p.~111]{HicksSharma1979} makes every $K_n$ a compact metrizable $z$-space, hence countable by \cite[Corollary~3]{HicksSharma1979}. Therefore $X$ is countable.
\end{proof}

The interval case will be used below.

\begin{corollary}\label{cor:no-square-refinement}
There is a neighborhood assignment $\varphi$ on $[0,1]$ for which no pair of neighborhood assignments $M,N$ on $[0,1]$ satisfies
\[
M[N(x)]\subseteq\varphi(x)
\qquad\text{for every }x\in[0,1].
\]
\end{corollary}

\begin{proof}
The space $[0,1]$ is uncountable, compact, and metrizable. Hence condition \textup{(i)} of Corollary~\ref{cor:compact-metric-roots} fails, so condition \textup{(ii)} fails as well. Thus some neighborhood assignment $\varphi$ has no pair $M,N$ with $M[N(x)]\subseteq\varphi(x)$ for every $x\in[0,1]$.
\end{proof}

\subsection{Lifts through continuous closed surjections}

A related fiber-saturation construction appears in the proof of \cite[Proposition~2.13]{KunziFletcher1985}. For continuous closed surjections, the fixed-assignment statement takes the form below.

\Needspace{12\baselineskip}
\begin{theorem}\label{thm:closed-lift}
Let $p:Z\to X$ be a continuous closed surjection and let $\varphi$ be a neighborhood assignment on $X$. The following are equivalent.
\begin{enumerate}[label=\textup{(\roman*)}]
\item There is a neighborhood assignment $V$ on $Z$ such that
\[
p\bigl(V[V(z)]\bigr)\subseteq\varphi(p(z))
\qquad(z\in Z).
\]
\item There is a neighborhood assignment $M$ on $X$ such that
\[
M[M(x)]\subseteq\varphi(x)
\qquad(x\in X).
\]
\end{enumerate}
\end{theorem}

\begin{proof}
Assume \textup{(i)}, and choose a neighborhood assignment $V$ witnessing it. For $x\in X$, write
\[
F_x=p^{-1}(\{x\})
\]
and define
\begin{equation}\label{eq:downstairs-M}
M(x)=X\setminus p\bigl(Z\setminus V[F_x]\bigr).
\end{equation}
Because $V[F_x]$ is open and $p$ is closed, $M(x)$ is open. Since $F_x\subseteq V[F_x]$, no point of $Z\setminus V[F_x]$ maps to $x$, so $x\in M(x)$.

Fix $y\in M(x)$. By \eqref{eq:downstairs-M}, the whole fiber $F_y$ is contained in $V[F_x]$. If $w\in F_y$, choose $t\in F_x$ with $w\in V(t)$. Then
\[
V(w)\subseteq V[V(t)],
\]
so
\[
p(V(w))\subseteq p(V[V(t)])\subseteq\varphi(p(t))=\varphi(x).
\]
Taking the union over $w\in F_y$ gives
\[
p(V[F_y])\subseteq\varphi(x).
\]
For the inclusion $M(y)\subseteq p(V[F_y])$, let $u\in M(y)$. By surjectivity choose $v\in Z$ with $p(v)=u$. If $v\notin V[F_y]$, then $u\in p(Z\setminus V[F_y])$, contradicting the definition of $M(y)$. Hence $v\in V[F_y]$, so $u\in p(V[F_y])$. Thus
\[
M(y)\subseteq p(V[F_y])\subseteq\varphi(x)
\]
whenever $y\in M(x)$. Taking the union over $y\in M(x)$ gives
\[
M[M(x)]\subseteq\varphi(x).
\]
This proves \textup{(ii)}.

Conversely, assume \textup{(ii)}, and choose a neighborhood assignment $M$ witnessing it. Define
\[
V(z)=p^{-1}\bigl(M(p(z))\bigr).
\]
By continuity of $p$, the map $V$ is a neighborhood assignment on $Z$. Surjectivity gives
\[
p(V(w))=M(p(w))\qquad(w\in Z)
\]
and, in particular, $p(V(z))=M(p(z))$. Therefore
\[
p(V[V(z)])
 =M[p(V(z))]
 =M[M(p(z))]
 \subseteq\varphi(p(z)),
\]
which proves \textup{(i)}.
\end{proof}

\begin{corollary}\label{cor:no-closed-lift}
For the assignment $\varphi$ in Corollary~\ref{cor:no-square-refinement}, no continuous closed surjection $p:Z\to[0,1]$ admits a neighborhood assignment $V$ on $Z$ satisfying
\[
p(V[V(z)])\subseteq\varphi(p(z))
\qquad(z\in Z).
\]
\end{corollary}

\begin{proof}
Otherwise Theorem~\ref{thm:closed-lift} would give a neighborhood assignment $M$ on $[0,1]$ with $M[M(x)]\subseteq\varphi(x)$ for every $x$, contradicting Corollary~\ref{cor:no-square-refinement}.
\end{proof}

\begin{remark}\label{rem:scope-lift}
For a continuous closed surjection $p:Z\to X$ and a fixed neighborhood assignment $\varphi$ on $X$, Theorem~\ref{thm:closed-lift} shows that a neighborhood assignment $V$ on $Z$ satisfying $p(V[V(z)])\subseteq\varphi(p(z))$ for every $z\in Z$ can exist only if $\varphi$ already has a pointwise square refinement on $X$. This does not address constructions that control only those centers that are eventually selected or other lifting arguments.
\end{remark}

\section{\texorpdfstring{An $\mathrm{HFC}_{\mathrm w}$ obstruction to charmingness}{An HFCw obstruction to charmingness}}\label{sec:hfcw}

Uncountable $\mathrm{HFC}_{\mathrm w}$ subspaces of $[\omega_1]^{\leq\omega}$ serve as a comparison class for the charming-space $D$-problem. Soukup and Szeptycki studied HFC and $\mathrm{HFC}_{\mathrm w}$ spaces in connection with the $D$-property \cite[Theorems~1, 4 and~7]{SoukupSzeptycki2019}. For such subspaces, the $\mathrm{HFC}_{\mathrm w}$ covering condition has two consequences used below. It rules out uncountable Lindel\"of $\Sigma$ subspaces and, when applied to pairwise disjoint families of singletons, forces countable sets to have countable closure. The first consequence forces any Lindel\"of $\Sigma$ kernel to be countable. A second singleton-family application then rules out charmingness. Combining the countable-closure consequence with hereditary Lindel\"ofness gives closed countable assignment kernels. The $\diamondsuit$-construction in version~2 of Tsaban's preprint falls within this class \cite[version~2, Section~4]{Tsaban2026}.

For $0<m<\omega$, let $F=\{F(0)<\cdots<F(m-1)\}\in[\omega_1]^m$. For $b\subseteq m$, put
\[
F[b]=\{F(i):i\in b\}.
\]
If $\mathcal F\subseteq[\omega_1]^m$, write
\[
[b,\mathcal F]
 =\{x\subseteq\omega_1:(\exists F\in\mathcal F)\ x\cap F=F[b]\}.
\]
Using the formulation in \cite[Definition~3.2(ii)]{Juhasz2002}, an uncountable subspace $Z\subseteq\mathcal P(\omega_1)$ is \emph{weakly hereditarily finally covering} ($\mathrm{HFC}_{\mathrm w}$) if for every $0<m<\omega$ and every uncountable pairwise disjoint family $\mathcal F\subseteq[\omega_1]^m$, there is a countable $\mathcal F_0\subseteq\mathcal F$ such that
\[
|Z\setminus[b,\mathcal F_0]|\leq\omega
\qquad(b\subseteq m).
\]
This property passes to every uncountable subspace $Y\subseteq Z$. For fixed $m$ and an uncountable pairwise disjoint family $\mathcal F\subseteq[\omega_1]^m$, any countable $\mathcal F_0\subseteq\mathcal F$ witnessing the condition for $Z$ also witnesses it for $Y$, since $Y\setminus[b,\mathcal F_0]\subseteq Z\setminus[b,\mathcal F_0]$ for every $b\subseteq m$.

\begin{theorem}\label{thm:hfcw-charming}
Let $Z\subseteq[\omega_1]^{\leq\omega}$ be an uncountable $\mathrm{HFC}_{\mathrm w}$ subspace of the Cantor cube $\mathcal P(\omega_1)$. Then the following statements hold.
\begin{enumerate}[label=\textup{(\roman*)}]
\item Every Lindel\"of $\Sigma$ subspace of $Z$ is countable.
\item The closure in $Z$ of every countable subset of $Z$ is countable.
\item $Z$ is not charming.
\item $Z$ is dually closed countable in the sense of \cite[Section~4]{YuYun2009}, meaning that every neighborhood assignment on $Z$ has a closed countable kernel.
\end{enumerate}
\end{theorem}

\begin{proof}
For \textup{(i)}, suppose that $Y\subseteq Z$ is an uncountable Lindel\"of $\Sigma$ subspace. Fix a compact cover $\mathcal C$ of $Y$ and a countable network $\mathcal N=\{N_n:n\in\omega\}$ of $Y$ with respect to $\mathcal C$. Choose distinct points $y_\xi\in Y$ for $\xi<\omega_1$, and partition $\omega_1$ into pairwise disjoint sets $A_\xi$, $\xi<\omega_1$, each of cardinality $\omega_1$.

For each $\xi<\omega_1$, choose $K_\xi\in\mathcal C$ with $y_\xi\in K_\xi$. Every member of $K_\xi$ is a countable subset of $\omega_1$, whereas $A_\xi$ is uncountable. Hence
\[
\bigl\{K_\xi\cap\{z\in Z:\alpha\notin z\}:\alpha\in A_\xi\bigr\}
\]
is an open cover of $K_\xi$. By compactness there is a nonempty finite set $F_\xi\subseteq A_\xi$ such that
\[
K_\xi\cap R_{F_\xi}=\varnothing,
\qquad
R_F:=\{z\in Z:F\subseteq z\}.
\]
The set $R_{F_\xi}$ is clopen in $Z$, so $Y\setminus R_{F_\xi}$ is an open neighborhood of $K_\xi$ in $Y$. Choose $n(\xi)\in\omega$ with
\[
K_\xi\subseteq N_{n(\xi)}\subseteq Y\setminus R_{F_\xi}.
\]
After thinning the index set, there exist an uncountable set $I\subseteq\omega_1$, an $n\in\omega$, and an integer $m>0$ such that
\[
n(\xi)=n
\quad\text{and}\quad
|F_\xi|=m
\qquad(\xi\in I).
\]
Because $F_\xi\subseteq A_\xi$, the family
\[
\mathcal F=\{F_\xi:\xi\in I\}\subseteq[\omega_1]^m
\]
is uncountable and pairwise disjoint. The uncountable subspace $Y$ is itself $\mathrm{HFC}_{\mathrm w}$, so there is a countable $\mathcal F_0\subseteq\mathcal F$ such that
\[
|Y\setminus[b,\mathcal F_0]|\leq\omega
\qquad(b\subseteq m).
\]
Take $b=\{0,\ldots,m-1\}$. Then
\[
Z\cap[b,\mathcal F_0]
 =\bigcup_{F\in\mathcal F_0}R_F.
\]
The set $N_n$ is disjoint from $R_{F_\xi}$ for every $\xi\in I$, and therefore
\[
N_n\subseteq Y\setminus[b,\mathcal F_0].
\]
Thus $N_n$ is countable. On the other hand, $y_\xi\in K_\xi\subseteq N_n$ for every $\xi\in I$, which is impossible. This proves \textup{(i)}.

For \textup{(ii)}, let $A\subseteq Z$ be countable and put
\[
S=\bigcup A.
\]
The set $S$ is countable. Choose an uncountable set $B\subseteq\omega_1\setminus S$ and apply the $\mathrm{HFC}_{\mathrm w}$ property to the pairwise disjoint family
\[
\mathcal F=\{\{\beta\}:\beta\in B\}\subseteq[\omega_1]^1.
\]
For $b=\{0\}$ there is a countable $\mathcal F_0\subseteq\mathcal F$ such that
\[
C:=Z\setminus[b,\mathcal F_0]
\]
is countable. The set $[b,\mathcal F_0]\cap Z$ is open in $Z$, so $C$ is closed in $Z$. Since every member of $A$ is disjoint from every singleton in $\mathcal F_0$, one has $A\subseteq C$. Consequently,
\[
\overline A^{\,Z}\subseteq C,
\]
and \textup{(ii)} follows.

For \textup{(iii)}, suppose that $Z$ is charming and let $Y$ be a Lindel\"of $\Sigma$ kernel. By \textup{(i)}, $Y$ is countable, so $S=\bigcup Y$ is countable. For each $\beta\in\omega_1\setminus S$, put
\[
R_\beta=\{z\in Z:\beta\in z\}.
\]
The set $R_\beta$ is closed in $Z$ and disjoint from $Y$. Hence it is Lindel\"of $\Sigma$ by the defining property of the kernel $Y$, and therefore countable by \textup{(i)}. Apply the $\mathrm{HFC}_{\mathrm w}$ property to the uncountable pairwise disjoint family
\[
\mathcal F=\{\{\beta\}:\beta\in\omega_1\setminus S\}
\]
and the pattern $b=\{0\}$. There is a countable set $S_0\subseteq\omega_1\setminus S$ such that
\[
Z\setminus\bigcup_{\beta\in S_0}R_\beta
\]
is countable. The union $\bigcup_{\beta\in S_0}R_\beta$ is countable, and its complement in $Z$ is countable by the preceding display. Thus $Z$ is countable, a contradiction.

For \textup{(iv)}, every $\mathrm{HFC}_{\mathrm w}$ space is hereditarily Lindel\"of \cite[p.~3]{SoukupSzeptycki2019}, hence Lindel\"of. Let $M$ be a neighborhood assignment on $Z$ and choose a countable kernel $C$ for $M$. By \textup{(ii)}, the closure $\overline C^{\,Z}$ is countable. Since the closure of a kernel is again a kernel, $\overline C^{\,Z}$ is a closed countable kernel for $M$.
\end{proof}

Soukup and Szeptycki observed that their particular $\diamondsuit$-constructed $\mathrm{HFC}_{\mathrm w}$ example is dually second countable \cite[Section~2]{SoukupSzeptycki2019}. By contrast, Theorem~\ref{thm:hfcw-charming}\textup{(iv)} gives a closed countable assignment kernel for every uncountable $\mathrm{HFC}_{\mathrm w}$ subspace of $[\omega_1]^{\leq\omega}$.

Tsaban also states this countable-closure property for his version~2 construction and notes that the construction, together with its $\mathrm{HFC}_{\mathrm w}$ and non-$D$ features, answers both parts of Soukup and Szeptycki's Problem~15 and the $\sigma$-compact-closure variant following that problem \cite[version~2, Section~5]{Tsaban2026}. The proof of Theorem~\ref{thm:hfcw-charming}\textup{(ii)} derives the same countable-closure conclusion for every uncountable $\mathrm{HFC}_{\mathrm w}$ subspace of $[\omega_1]^{\leq\omega}$ from the $m=1$ case of the $\mathrm{HFC}_{\mathrm w}$ condition, rather than from features peculiar to Tsaban's construction.

\begin{corollary}\label{cor:tsaban-not-charming}
The space constructed under $\diamondsuit$ in version~2 of Tsaban's preprint \cite[version~2, Theorem~1.1]{Tsaban2026} is dually closed countable and not charming.
\end{corollary}

\begin{proof}
Tsaban's version~2 construction in \cite[version~2, Section~3]{Tsaban2026} has the form $X=\{x_\xi:\xi<\omega_1\}$ with $x_\xi\subseteq\xi$, so $X\subseteq[\omega_1]^{\leq\omega}$. By \cite[version~2, Theorem~1.1]{Tsaban2026}, this space has cardinality $\aleph_1$ and is therefore uncountable. Tsaban \cite[version~2, Section~4]{Tsaban2026} notes that the proof of Theorem~1.1 in that version shows that $X$ is $\mathrm{HFC}_{\mathrm w}$. Theorem~\ref{thm:hfcw-charming}\textup{(iii)} therefore rules out charmingness independently of the announced non-$D$ property, while \textup{(iv)} gives dually closed countability.
\end{proof}

A closed countable assignment kernel is Lindel\"of $\Sigma$, and every countable space is a $D$-space. Hence Theorem~\ref{thm:hfcw-charming}\textup{(iv)} gives the conclusions of Theorem~\ref{thm:closed-LSigma-kernel} and Corollary~\ref{cor:two-assignments} for every such $\mathrm{HFC}_{\mathrm w}$ space. Tsaban's version~2 construction therefore shows that, under $\diamondsuit$, there is a dually closed countable space that is neither charming nor a $D$-space. The remaining issue for charming spaces is the one-step condition in Theorem~\ref{thm:exact-criterion}\textup{(ii)}.

\section{The closed-core reduction}\label{sec:boundary-iteration}

\subsection{\texorpdfstring{The boundary obstruction and the $G_\delta$ criterion}{The boundary obstruction and the G-delta criterion}}

The boundary obstruction in \cite[Definition~3.30]{HuYin2026} admits a direct description. The underlying observation holds for every Lindel\"of subspace of a Hausdorff space.

\begin{lemma}\label{lem:cc-closure}
Let $X$ be Hausdorff, let $Y\subseteq X$ be Lindel\"of, and let $A\subseteq Y$ be countably compact. Then
\[
\overline A^{\,X}\subseteq Y.
\]
\end{lemma}

\begin{proof}
If $A=\varnothing$, there is nothing to prove. Assume $A\ne\varnothing$ and fix $x\in X\setminus Y$. For each $y\in Y$, choose disjoint open sets $O_y$ and $V_y$ with $y\in O_y$ and $x\in V_y$. The family $\{O_y\cap Y:y\in Y\}$ covers the Lindel\"of space $Y$, so there is a countable set $C\subseteq Y$ such that $Y\subseteq\bigcup_{y\in C} O_y$. The sets $O_y\cap A$, $y\in C$, form a countable open cover of the countably compact space $A$, so for some finite $E\subseteq C$ one has
\[
A\subseteq\bigcup_{y\in E}O_y.
\]
Put
\[
V=\bigcap_{y\in E}V_y.
\]
Then $V$ is an open neighborhood of $x$, and $V\cap O_y=\varnothing$ for every $y\in E$. Hence $V\cap A=\varnothing$, so $x\notin\overline A^{\,X}$. Since $x\in X\setminus Y$ was arbitrary, $\overline A^{\,X}\subseteq Y$.
\end{proof}

\begin{proposition}\label{prop:all-boundary-bad}
Let $X$ be Hausdorff, let $Y\subseteq X$ be Lindel\"of, and put
\[
B=\overline Y^{\,X}.
\]
If $x\in B\setminus Y$ and $U$ is any open neighborhood of $x$ in $X$, then $U\cap Y$ is not countably compact.
\end{proposition}

\begin{proof}
Because $x\in\overline Y^{\,X}$, every open neighborhood $W$ of $x$ has
\[
(W\cap U)\cap Y\ne\varnothing.
\]
Hence
\[
x\in\overline{U\cap Y}^{\,X}.
\]
If $U\cap Y$ were countably compact, Lemma~\ref{lem:cc-closure} would imply that its closure in $X$ is contained in $Y$, contradicting $x\notin Y$.
\end{proof}

For a charming space $X$ with Lindel\"of $\Sigma$ kernel $Y$ and $B=\overline Y^{\,X}$, the earlier obstruction set \cite[Definition~3.30]{HuYin2026} was defined by
\[
\begin{split}
\Obs(Y,B)=\{x\in B\setminus Y:\ &U\cap Y\text{ is not countably compact}\\
&\text{for every open neighborhood }U\text{ of }x\}.
\end{split}
\]
The corresponding closed core \cite[Definition~3.34]{HuYin2026} was
\[
\Core(Y,B)=\overline{\Obs(Y,B)}^{\,B}.
\]
The following identities will be used in the closed-core analysis.

\begin{corollary}\label{cor:obs-identity}
Let $X$ be a charming space with Lindel\"of $\Sigma$ kernel $Y$, and put $B=\overline Y^{\,X}$. Then
\[
\Obs(Y,B)=B\setminus Y
\]
and
\[
\Core(Y,B)=\overline{B\setminus Y}^{\,B}
=B\setminus\operatorname{Int}_B(Y).
\]
\end{corollary}

\begin{proof}
Proposition~\ref{prop:all-boundary-bad} gives $\Obs(Y,B)=B\setminus Y$. By \cite[Definition~3.34]{HuYin2026}, one has $\Core(Y,B)=\overline{B\setminus Y}^{\,B}$. The closure identity $\overline{B\setminus Y}^{\,B}=B\setminus\operatorname{Int}_B(Y)$ gives the second equality.
\end{proof}

These identities give equivalent $G_\delta$ and $F_\sigma$ formulations of the descriptive-complexity condition on the core trace.

\Needspace{12\baselineskip}
\begin{proposition}\label{prop:gdelta-equivalence}
Under the hypotheses of Corollary~\ref{cor:obs-identity}, write
\[
H=\Core(Y,B)=\overline{B\setminus Y}^{\,B}.
\]
Then the following are equivalent.
\begin{enumerate}[label=\textup{(\roman*)}]
\item $H\cap Y$ is a $G_\delta$ subset of $H$.
\item $B\setminus Y$ is an $F_\sigma$ subset of $B$.
\item $Y$ is a $G_\delta$ subset of $B$.
\end{enumerate}
\end{proposition}

\begin{proof}
Since $B\setminus Y\subseteq H$, one has
\[
H\setminus(H\cap Y)=H\setminus Y=B\setminus Y.
\]
Thus \textup{(i)} is equivalent to $B\setminus Y$ being $F_\sigma$ in $H$. The set $H$ is closed in $B$, so every set closed in $H$ is closed in $B$. Hence an $F_\sigma$ representation in $H$ is also one in $B$. Conversely, an $F_\sigma$ representation in $B$ restricts to $H$. This proves the equivalence of \textup{(i)} and \textup{(ii)}. The equivalence of \textup{(ii)} and \textup{(iii)} follows by complementation in $B$.
\end{proof}

Li, Lin and Lin call a charming space \emph{CL charming} when it has a closed Lindel\"of $\Sigma$ kernel \cite[Definition~4.1]{LiLinLin2017}. Every CL charming space is a $D$-space \cite[Corollary~3.4]{HuYin2026}. Proposition~\ref{prop:gdelta-equivalence} shows that the core-trace condition is equivalent to $Y$ being $G_\delta$ in $B=\overline{Y}^{\,X}$. Combining this equivalence with the Lindel\"of $\Sigma$ closure calculation in the proof of \cite[Proposition~3.10]{HuYin2026} gives the characterization below.

\begin{theorem}\label{thm:CL-equivalence}
Let $X$ be a charming space. The following are equivalent.
\begin{enumerate}[label=\textup{(\roman*)}]
\item $X$ is CL charming.
\item There is a Lindel\"of $\Sigma$ kernel $Y$ of $X$ such that, with $B=\overline Y^{\,X}$ and $H=\overline{B\setminus Y}^{\,B}$, the set $H\cap Y$ is $G_\delta$ in $H$.
\item There is a Lindel\"of $\Sigma$ kernel $Y$ of $X$ such that, with $B=\overline Y^{\,X}$ and $H=\overline{B\setminus Y}^{\,B}$, the set $H\cap Y$ is both dense and $G_\delta$ in $H$.
\end{enumerate}
\end{theorem}

\begin{proof}
The implication \textup{(iii)}$\Rightarrow$\textup{(ii)} follows by dropping the density requirement.
Assume \textup{(ii)}, and choose a Lindel\"of $\Sigma$ kernel $Y$ witnessing it. Put $B=\overline Y^{\,X}$ and $H=\overline{B\setminus Y}^{\,B}$. Proposition~\ref{prop:gdelta-equivalence} shows that $Y$ is a $G_\delta$ subset of $B$. The closure calculation in the proof of \cite[Proposition~3.10]{HuYin2026} then gives that $B$ is Lindel\"of $\Sigma$. Since $B$ is closed in $X$, it remains only to check the kernel condition. If $U$ is open in $X$ and $B\subseteq U$, then $Y\subseteq U$, so $X\setminus U$ is Lindel\"of $\Sigma$ by the defining property of the kernel $Y$. Hence $B$ is a closed Lindel\"of $\Sigma$ kernel of $X$, and $X$ is CL charming.

For the remaining implication, assume \textup{(i)} and choose a closed Lindel\"of $\Sigma$ kernel $Y$. Put $B=\overline Y^{\,X}$ and $H=\overline{B\setminus Y}^{\,B}$. Then $B=Y$, so $H=\varnothing$. The empty subset of the empty space is both dense and $G_\delta$. Thus \textup{(iii)} holds.
\end{proof}

\begin{remark}\label{rem:old-core}
Under the notation of Corollary~\ref{cor:obs-identity}, the closed set in \cite[Theorem~3.36]{HuYin2026} is $\overline{B\setminus Y}^{\,B}$. The earlier theorem therefore states that $X$ is a $D$-space if and only if this boundary closure is a $D$-space, which is the first-stage case of Theorem~\ref{thm:core-iteration}. This identity also simplifies several earlier boundary criteria. The hypothesis of \cite[Proposition~3.24]{HuYin2026} can hold only when $B\setminus Y=\varnothing$. In \cite[Theorem~3.25 and Corollary~3.32]{HuYin2026}, the exceptional closed set must contain all of $B\setminus Y$, so those criteria reduce to the $G_\delta$-kernel situation handled in \cite[Proposition~3.10]{HuYin2026}. Proposition~\ref{prop:gdelta-equivalence} shows that \cite[Proposition~3.38]{HuYin2026} is equivalent to that $G_\delta$-kernel criterion. Theorem~\ref{thm:CL-equivalence} also shows that \cite[Conjecture~3.42]{HuYin2026} is equivalent to the assertion that every charming space is CL charming. The present paper does not decide this reformulated assertion.
\end{remark}

\subsection{The transfinite closed-core iteration}

This reduction can be iterated while the original kernel remains fixed. The underlying trace derivative is classical. For a topological space $X$ and $A\subseteq X$, write
\[
\rho_X(A)=A\cap\overline{\,\overline A^{\,X}\setminus A\,}^{\,X}
\]
for the Hausdorff residue of $A$ \cite[p.~281]{Hausdorff1914}. Kuratowski \cite[Section~12]{Kuratowski1966} treats the transfinite residue theory.

\begin{definition}\label{def:core-operator}
Let $X$ be a charming space with a fixed Lindel\"of $\Sigma$ kernel $Y$. For each closed set $F\subseteq X$, define
\[
T(F)=\overline{\,\overline{Y\cap F}^{\,X}\setminus Y\,}^{\,X}.
\]
\end{definition}

The fixed-kernel trace of $T$ is the Hausdorff residue, and $T$ also preserves finite unions. The following proposition states these operator properties together with the fixed-point criterion.

\begin{proposition}\label{prop:core-fixed-points}
Under the notation of Definition~\ref{def:core-operator}, the operator $T$ is monotone and satisfies $T(F)\subseteq F$ for every closed $F\subseteq X$. For closed sets $F,G\subseteq X$,
\[
T(F\cup G)=T(F)\cup T(G).
\]
If $Y\cap F=Y\cap G$, then $T(F)=T(G)$. Moreover, for every closed $F\subseteq X$,
\[
Y\cap T(F)=\rho_X(Y\cap F).
\]
The equality $T(F)=F$ holds if and only if both $Y\cap F$ and $F\setminus Y$ are dense in $F$.
\end{proposition}

\begin{proof}
Monotonicity and trace dependence follow directly from the definition. Let $F\subseteq X$ be closed and put $A=Y\cap F$. Since $\overline A^{\,X}\subseteq F$, one has $T(F)\subseteq F$ and $\overline A^{\,X}\cap Y=A$. Hence $\overline A^{\,X}\setminus Y=\overline A^{\,X}\setminus A$ and $Y\cap T(F)=A\cap T(F)$, so
\[
Y\cap T(F)=A\cap\overline{\,\overline A^{\,X}\setminus A\,}^{\,X}=\rho_X(A).
\]
For closed $F,G\subseteq X$, finite unions commute with the two closure operations and with subtraction of the fixed set $Y$, which gives
\[
T(F\cup G)=T(F)\cup T(G).
\]
If $T(F)=F$, then
\[
F=T(F)\subseteq\overline{Y\cap F}^{\,X}\subseteq F,
\]
so $Y\cap F$ is dense in $F$. The set $\overline{Y\cap F}^{\,X}\setminus Y\subseteq F\setminus Y$ has closure $T(F)=F$, so $F\setminus Y$ is dense in $F$. Conversely, if both sets are dense in $F$, then $\overline{Y\cap F}^{\,X}=F$ and therefore
\[
T(F)=\overline{F\setminus Y}^{\,X}=F.
\]
\end{proof}

Consequently, if the operator $T$ is iterated by applying $T$ at successor stages and taking intersections at nonzero limit stages, the traces follow Hausdorff's residue iteration. At each successor stage the trace is replaced by its $\rho_X$-residue. At every nonzero limit stage the preceding traces are intersected. The operator $T$ is the closed ambient counterpart that links this trace iteration to the charming-space $D$-reduction.

Closed-subspace inheritance for charming spaces is the $(L\Sigma,L\Sigma)$-case of \cite[Proposition~3.2\textup{(2)}]{CasarrubiasPaniagua2016}. The lemma below combines the inherited-kernel statement with the additional empty-core consequence needed in the iteration.

\begin{lemma}\label{lem:closed-restriction-charming}
Under the notation of Definition~\ref{def:core-operator}, if $F$ is closed in $X$, then $F$ is charming with Lindel\"of $\Sigma$ kernel $Y\cap F$. If, in addition, $T(F)=\varnothing$, then $Y\cap F$ is closed in $F$, so $F$ is CL charming and hence a $D$-space.
\end{lemma}

\begin{proof}
If $Y\cap F\ne\varnothing$, the proof of \cite[Proposition~3.2\textup{(2)}]{CasarrubiasPaniagua2016} gives $Y\cap F$ as a Lindel\"of $\Sigma$ kernel of $F$. If $Y\cap F=\varnothing$, then $X\setminus F$ is an open neighborhood of $Y$, so $F$ is Lindel\"of $\Sigma$. For every open $V\subseteq F$, the set $F\setminus V$ is closed in $F$ and hence Lindel\"of $\Sigma$. Therefore $\varnothing=Y\cap F$ is a Lindel\"of $\Sigma$ kernel of $F$. Thus $F$ is charming with kernel $Y\cap F$.

Now suppose $T(F)=\varnothing$. Then
\[
\overline{Y\cap F}^{\,X}\setminus Y=\varnothing.
\]
Since $F$ is closed, $\overline{Y\cap F}^{\,X}\subseteq F$, and therefore
\[
\overline{Y\cap F}^{\,X}=Y\cap F.
\]
Hence $Y\cap F$ is closed in $F$. It is thus a closed Lindel\"of $\Sigma$ kernel of $F$, so $F$ is CL charming and consequently a $D$-space.
\end{proof}

The limit step uses the classical countable closed-addition theorem for $D$-spaces. Guo and Junnila proved a transfinite addition theorem for covers by $D$-subspaces with closed initial unions \cite{GuoJunnila2011}. The argument below needs only the countable closed-addition case and the following elementary intersection lemma.

\begin{lemma}\label{lem:closed-intersection}
Let $Z$ be a regular Lindel\"of space, and let $\{H_i:i\in I\}$ be a family of closed subsets of $Z$. Suppose that, for every $i\in I$, each closed subset of $Z$ disjoint from $H_i$ is a $D$-space. Then every closed subset of $Z$ disjoint from $\bigcap_{i\in I}H_i$ is a $D$-space.
\end{lemma}

\begin{proof}
A countable union of closed $D$-subspaces is a $D$-space \cite[Proposition~7]{BorgesWehrly1991}. Let $C$ be closed in $Z$ and disjoint from $\bigcap_{i\in I}H_i$. If $C=\varnothing$, there is nothing to prove. For each $x\in C$, choose $i(x)\in I$ with $x\notin H_{i(x)}$. By regularity, there is an open neighborhood $V_x$ of $x$ in $C$ with
\[
\overline{V_x}^{\,C}\subseteq C\setminus H_{i(x)}.
\]
The Lindel\"of property of $C$ gives a countable set $A\subseteq C$ such that $\{V_x:x\in A\}$ covers $C$. For each $x\in A$, the set $\overline{V_x}^{\,C}$ is closed in $Z$ and misses $H_{i(x)}$, hence is a $D$-space. These closed subspaces cover $C$, so the cited countable closed-addition theorem shows that $C$ is a $D$-space.
\end{proof}

The following is the closed-set formulation of the gluing lemma of Hu and Yin. Its hypothesis is equivalent to theirs by taking complements of open neighborhoods of $K$.

\begin{lemma}[{\cite[Lemma~3.1]{HuYin2026}}]\label{lem:d-gluing}
Let $K$ be a closed $D$-subspace of a space $Z$. If every closed subset of $Z$ disjoint from $K$ is a $D$-space, then $Z$ is a $D$-space.
\end{lemma}

For a space $Z$, let $hL(Z)$ denote the hereditary Lindel\"of number, namely the least infinite cardinal $\mu$ such that every open cover of every subspace of $Z$ has a subcover of size at most $\mu$ \cite[p.~183]{Good2003}. Thus $hL(Z)\leq\max\{\omega,|Z|\}$, and $hL(Z)=\omega$ exactly when $Z$ is hereditarily Lindel\"of. The stabilization estimate below comes directly from this definition. In the proof, it is applied to the increasing family of portions of the fixed kernel removed along the closed-core iteration.

\Needspace{15\baselineskip}
\begin{theorem}\label{thm:core-iteration}
Let $X$ be a charming space with a fixed Lindel\"of $\Sigma$ kernel $Y$, and let $T$ be the operator from Definition~\ref{def:core-operator}. Put $B=\overline Y^{\,X}$. Then $T(X)=\overline{B\setminus Y}^{\,X}=\Core(Y,B)$ because $B$ is closed in $X$. Thus the first iterate is exactly the closed core $\Core(Y,B)$ from Corollary~\ref{cor:obs-identity}. Set $F_0=X$, $F_{\alpha+1}=T(F_\alpha)$, and
\[
F_\lambda=\bigcap_{\alpha<\lambda}F_\alpha
\qquad(\lambda\text{ a nonzero limit ordinal}).
\]
For every ordinal $\alpha$, the space $X$ is a $D$-space if and only if $F_\alpha$ is a $D$-space. The sequence stabilizes at some $\gamma<hL(Y)^+$, and hence at some $\gamma<(\max\{\omega,|Y|\})^+$. Its terminal value $F_*=F_\gamma=F_{\gamma+1}$ is the largest closed subset $F\subseteq X$ for which both $Y\cap F$ and $F\setminus Y$ are dense in $F$.

If $Y$ is hereditarily Lindel\"of, then $\gamma$ may be chosen countable. If $X$ is hereditarily Lindel\"of, then for every ordinal $\alpha$, the space $X\setminus F_\alpha$ is Lindel\"of $\Sigma$, and $X$ is Lindel\"of $\Sigma$ if and only if $F_\alpha$ is Lindel\"of $\Sigma$.
\end{theorem}

\begin{proof}
The space $X$ is Lindel\"of by the observation in Section~\ref{sec:preliminaries}. By Proposition~\ref{prop:core-fixed-points}, the operator $T$ is monotone and satisfies $T(F)\subseteq F$ for every closed $F\subseteq X$. Hence, whenever $C\subseteq F\subseteq X$ are closed,
\[
T(C)\subseteq C\cap T(F).
\]
If $C\subseteq X$ is closed and $T(C)=\varnothing$, Lemma~\ref{lem:closed-restriction-charming} shows that $C$ is a $D$-space.

A transfinite induction shows that every closed subset of $X$ disjoint from $F_\alpha$ is a $D$-space. For $\alpha=0$, there is nothing to prove. At a successor stage, let $C$ be closed in $X$ and disjoint from $F_{\alpha+1}$. Put $K=C\cap F_\alpha$. Then
\[
T(K)\subseteq K\cap T(F_\alpha)=\varnothing,
\]
so $K$ is a $D$-space. Every closed subset of $C$ disjoint from $K$ is closed in $X$ and disjoint from $F_\alpha$, hence is a $D$-space by the induction hypothesis. Lemma~\ref{lem:d-gluing} gives that $C$ is a $D$-space. At a nonzero limit stage $\lambda$, Lemma~\ref{lem:closed-intersection}, applied to the family $\{F_\alpha:\alpha<\lambda\}$, gives the induction statement.

Each $F_\alpha$ is closed in $X$, so Lemma~\ref{lem:closed-hereditary-D} shows that the $D$-property of $X$ implies that of $F_\alpha$, while the converse follows from the induction statement and Lemma~\ref{lem:d-gluing}. Put $\mu=hL(Y)$ and
\[
U=\bigcup_{\alpha<\mu^+}(Y\setminus F_\alpha).
\]
The cover $\{Y\setminus F_\alpha:\alpha<\mu^+\}$ of $U$ has a subcover indexed by some $S\subseteq\mu^+$ with $|S|\leq\mu$. Take $S=\varnothing$ when $U=\varnothing$. Put $\beta=\sup S$ (so $\beta=0$ if $S=\varnothing$). Regularity of $\mu^+$ gives $\beta<\mu^+$. Since $\xi\leq\beta$ for every $\xi\in S$, monotonicity gives $U=\bigcup_{\xi\in S}(Y\setminus F_\xi)\subseteq Y\setminus F_\beta$, while the reverse inclusion follows from $\beta<\mu^+$ and the definition of $U$. Hence $U=Y\setminus F_\beta$. Since $\beta+1<\mu^+$, monotonicity gives $Y\setminus F_\beta\subseteq Y\setminus F_{\beta+1}\subseteq U=Y\setminus F_\beta$. Thus $Y\cap F_\beta=Y\cap F_{\beta+1}$. By the trace-dependence assertion in Proposition~\ref{prop:core-fixed-points}, $F_{\beta+1}=F_{\beta+2}$. Thus $\gamma=\beta+1<\mu^+$ is a stabilization stage, and the limit definition makes every subsequent term equal to $F_*=F_\gamma$. In particular, $F_*=T(F_*)$, so Proposition~\ref{prop:core-fixed-points} shows that both $Y\cap F_*$ and $F_*\setminus Y$ are dense in $F_*$. Conversely, if a closed set $F\subseteq X$ has both sets dense, the same proposition gives $T(F)=F$. Monotonicity and transfinite induction give $F\subseteq F_\alpha$ for every $\alpha$, so $F\subseteq F_*$.

If $Y$ is hereditarily Lindel\"of, then $hL(Y)=\omega$, so the stabilization stage $\gamma$ is countable.

Now suppose that $X$ is hereditarily Lindel\"of. Then $Y$ is hereditarily Lindel\"of, so $\gamma<\omega_1$ may be chosen. For a closed $F\subseteq X$, put $B_F=\overline{Y\cap F}^{\,X}$. For each $x\in F\setminus B_F$, regularity gives an open neighborhood $U_x$ in $F$ with $\overline{U_x}^{\,F}\subseteq F\setminus B_F$. Since $F\setminus B_F$ is Lindel\"of, choose a countable set $Q_0\subseteq F\setminus B_F$ such that $F\setminus B_F=\bigcup_{x\in Q_0}U_x$. Each $\overline{U_x}^{\,F}$ is closed in $X$ and disjoint from $Y$, hence Lindel\"of $\Sigma$. These closures still cover $F\setminus B_F$, so $F\setminus B_F$ is Lindel\"of $\Sigma$. Also, since $B_F\subseteq F$ and $B_F\setminus Y\subseteq T(F)$, one has $B_F\setminus T(F)=(Y\cap F)\setminus T(F)$. Put $O=(Y\cap F)\setminus T(F)$. The set $O$ is open in the regular hereditarily Lindel\"of space $Y\cap F$. For each $y\in O$, choose an open neighborhood $V_y$ in $Y\cap F$ such that $\overline{V_y}^{\,Y\cap F}\subseteq O$. Lindel\"ofness gives a countable set $Q_1\subseteq O$ such that $O=\bigcup_{y\in Q_1}V_y$. Hence $O=\bigcup_{y\in Q_1}\overline{V_y}^{\,Y\cap F}$. Since $Y\cap F$ is closed in the Lindel\"of $\Sigma$ space $Y$, each of these closures is Lindel\"of $\Sigma$, and countable stability gives that $O$ is Lindel\"of $\Sigma$. The equality $F\setminus T(F)=(F\setminus B_F)\cup O$ shows that $F\setminus T(F)$ is Lindel\"of $\Sigma$. Because the chain $(F_\xi)$ is decreasing, is continuous at nonzero limit stages, and is constant from stage $\gamma$ onward, for every $\alpha$ one has
\[
X\setminus F_\alpha
=\bigcup_{\xi<\min\{\alpha,\gamma\}}(F_\xi\setminus F_{\xi+1}).
\]
If $x\notin F_\alpha$, let $\eta$ be the least stage at which $x\notin F_\eta$. By continuity, $\eta$ cannot be a nonzero limit, so $\eta=\xi+1$ and $\xi<\alpha$. If $\alpha\geq\gamma$, constancy from stage $\gamma$ onward also gives $\xi<\gamma$. Hence $\xi<\min\{\alpha,\gamma\}$. Conversely, every $F_\xi\setminus F_{\xi+1}$ on the right is disjoint from $F_\alpha$ by monotonicity. This union is countable, so $X\setminus F_\alpha$ is Lindel\"of $\Sigma$. The final equivalence follows from closed heredity and finite unions.
\end{proof}

\begin{corollary}\label{cor:terminal-counterexample}
Let $X$ be a charming non-$D$ space with a fixed Lindel\"of $\Sigma$ kernel $Y$, and let $F_*$ be the terminal value in Theorem~\ref{thm:core-iteration}. Then $F_*$ is a nonempty closed charming non-$D$ subspace of $X$, and both $Y\cap F_*$ and $F_*\setminus Y$ are dense in $F_*$. In particular, every charming non-$D$ space contains a closed non-$D$ terminal core for the corresponding operator $T$.
\end{corollary}

\begin{proof}
Theorem~\ref{thm:core-iteration} gives that $F_*$ is closed and that $X$ is a $D$-space if and only if $F_*$ is a $D$-space. Hence $F_*$ is non-$D$ and therefore nonempty. Lemma~\ref{lem:closed-restriction-charming} shows that $F_*$ is charming with kernel $Y\cap F_*$. The density assertions follow from Theorem~\ref{thm:core-iteration}, equivalently from Proposition~\ref{prop:core-fixed-points}.
\end{proof}

\begin{remark}\label{rem:stable-core}
Terminality alone does not force failure of the $D$-property. A nonempty terminal core can be a $D$-space. For example, let $X=[0,1]$ and $Y=\mathbb{Q}\cap[0,1]$. The countable set $Y$ is Lindel\"of $\Sigma$, and it is a Lindel\"of $\Sigma$ kernel of $X$ because every closed subset of $X$ is compact. Both $Y$ and its complement are dense, so for the operator $T$ associated with this kernel one has $T(X)=X$, although $X$ is a compact $D$-space. Thus Theorem~\ref{thm:core-iteration} reduces the general problem to terminal cores without deciding their $D$-property.
\end{remark}

\subsection{Sharpness of the stabilization bound}

Hausdorff's residue operation is classical \cite[p.~281]{Hausdorff1914}, and Kuratowski treats residues of transfinite order in \cite[Section~12]{Kuratowski1966}. The construction below realizes this residue iteration within the closed-core reduction while keeping the ambient space compact and the kernel dense, $\sigma$-compact, and Lindel\"of $\Sigma$. It also prescribes both the cardinality and the hereditary Lindel\"of number of the kernel. The construction uses Cantor--Bendixson ranks of ordinal compacta. Caicedo and Hilton \cite[Definition~2.4]{CaicedoHilton2017} give the Cantor-normal-form description used below. For an iteration $(F_\alpha)$ as in Theorem~\ref{thm:core-iteration}, the \emph{first stabilization stage} is the least ordinal $\gamma$ such that $F_\gamma=F_{\gamma+1}$.

\Needspace{9\baselineskip}
\begin{theorem}\label{thm:uncountable-core-rank}
For every infinite cardinal $\kappa$ and every ordinal $\gamma<\kappa^+$, there is a compact space $X$ with a dense $\sigma$-compact Lindel\"of $\Sigma$ kernel $Y$ such that $hL(Y)=|Y|=\kappa$ and the iteration in Theorem~\ref{thm:core-iteration} first stabilizes at $\gamma$.
\end{theorem}

\begin{proof}
For an ordinal $\theta$, put $W_\theta=[0,\omega^\theta]$ and denote its Cantor--Bendixson derivatives by $W_\theta^{(\xi)}$. The rank of $0$ is $0$, whereas the rank of a positive point is the last exponent in its Cantor normal form. Thus every rank at most $\theta$ occurs, and $W_\theta^{(\theta+1)}=\varnothing$. For $0<\xi\leq\theta$, the derivative $W_\theta^{(\xi)}$ consists of the nonzero multiples of $\omega^\xi$ in $W_\theta$. For every $\xi\leq\theta$, the points of rank exactly $\xi$ are dense in $W_\theta^{(\xi)}$, since they are its isolated points and the space is scattered. Write each rank as $\lambda+n$, where $\lambda$ is zero or a limit ordinal and $n<\omega$. Let $Y_\theta$ consist of the points whose ranks have even finite part. It contains every isolated point of $W_\theta$, so it is dense.

Induction on $\theta$ shows that $Y_\theta$ is $\sigma$-compact. The space $W_0=\{0,1\}$ is finite. If $\theta=\beta+1$, then, apart from $0$ and the top point, $W_\theta$ is the union of the clopen intervals
\[
(\omega^\beta\cdot n,\omega^\beta\cdot(n+1)],\qquad n<\omega.
\]
Each interval is order-homeomorphic to $(0,\omega^\beta]$, with the last Cantor-normal-form exponent preserved. Its trace on $Y_\theta$ is therefore homeomorphic to $Y_\beta\setminus\{0\}$, which is closed since $0$ is isolated in $Y_\beta$. These traces are $\sigma$-compact by induction. Adding back $0$ preserves $\sigma$-compactness, as does adding the top point when $\omega^\theta\in Y_\theta$. This argument also covers $\beta=0$, when the intervals are singletons.

Now let $\theta$ be a nonzero limit ordinal and put $\tau=\operatorname{cf}(\theta)$. Fix a cofinal sequence $(e_\alpha)_{\alpha<\tau}$ in $\theta$ and construct a continuous strictly increasing sequence $(\theta_\alpha)_{\alpha\leq\tau}$ as follows. Set $\theta_0=0$. At a successor stage choose $\theta_{\alpha+1}$ above both $\theta_\alpha$ and $e_\alpha$ with positive even finite part. Such a choice is always possible. If $\eta<\theta$ and $\eta=\lambda+n$ with $\lambda$ zero or limit and $n<\omega$, take $\zeta=\eta+1$ when $n$ is odd and $\zeta=\eta+2$ when $n$ is even. Then $\eta<\zeta<\theta$ because $\theta$ is limit, and $\zeta$ has positive even finite part. At a nonzero limit index $\delta<\tau$, set $\theta_\delta=\sup_{\alpha<\delta}\theta_\alpha$, which is still below $\theta$ because $\delta<\operatorname{cf}(\theta)$. Set $\theta_\tau=\theta$. Since $\theta_{\alpha+1}>e_\alpha$ for every $\alpha<\tau$, the sequence below $\tau$ is cofinal in $\theta$, so this choice makes the sequence continuous at $\tau$. The set
\[
C=\{\omega^{\theta_\alpha}:\alpha\leq\tau\}
\]
is compact and therefore closed in $W_\theta$. It is the image of the compact ordinal interval $[0,\tau]$ under the continuous map $\alpha\mapsto\omega^{\theta_\alpha}$. Continuity follows from continuity of the sequence $(\theta_\alpha)_{\alpha\leq\tau}$ and of ordinal exponentiation $\beta\mapsto\omega^\beta$ at limit ordinals. Also $C\subseteq Y_\theta$. Here $\theta_0=0$, successor-indexed exponents have even finite part by construction, and nonzero limit-indexed exponents are limit ordinals and hence have finite part $0$. The complement $W_\theta\setminus C$ consists of the point $0$ together with the convex open intervals between consecutive members of $C$. Let $\mathcal J$ denote the family of these intervals. If $J=(a,b)\in\mathcal J$, then its right endpoint has the form $b=\omega^\beta$, where $\beta<\theta$ is a successor ordinal, so $\operatorname{cf}(b)=\omega$. Choose an increasing sequence $(c_m)_{m<\omega}$ cofinal in $b$, with $c_m\geq a+1$ for every $m<\omega$. Then
\[
J=\bigcup_{m<\omega}[a+1,c_m]
\]
is $F_\sigma$ in $W_\beta$. Since $Y_\theta\cap W_\beta=Y_\beta$, the set $Y_\theta\cap J=Y_\beta\cap J$ is $F_\sigma$ in the $\sigma$-compact space $Y_\beta$, and hence is $\sigma$-compact. For each $J\in\mathcal J$, choose an increasing sequence $(L_{J,n})_{n<\omega}$ of compact sets whose union is $Y_\theta\cap J$. Since $0\in Y_\theta$, for $n<\omega$ put
\[
K_n=C\cup\{0\}\cup\bigcup_{J\in\mathcal J} L_{J,n}.
\]
Each $L_{J,n}$ is compact, hence closed in $W_\theta$. Since $W_\theta\setminus(C\cup\{0\})=\bigcup_{J\in\mathcal J}J$, one has $W_\theta\setminus K_n=\bigcup_{J\in\mathcal J}(J\setminus L_{J,n})$, which is open. Hence $K_n$ is closed and therefore compact. The inclusions $C\cup\{0\}\subseteq Y_\theta$ and $L_{J,n}\subseteq Y_\theta$, together with $\bigcup_{n<\omega}L_{J,n}=Y_\theta\cap J$ for each $J\in\mathcal J$, give $K_n\subseteq Y_\theta$ and $Y_\theta=\bigcup_{n<\omega}K_n$, completing the induction. Thus $Y_\theta$ is $\sigma$-compact, hence Lindel\"of $\Sigma$. Since $W_\theta$ is compact, $W_\theta\setminus U$ is compact for every open neighborhood $U$ of $Y_\theta$ in $W_\theta$. Therefore $Y_\theta$ is a Lindel\"of $\Sigma$ kernel of $W_\theta$. For each $\theta$, let $T_\theta$ denote the operator from Definition~\ref{def:core-operator} associated with the pair $(W_\theta,Y_\theta)$, and let $(F^\theta_\alpha)$ denote the corresponding iteration from Theorem~\ref{thm:core-iteration}, with $F^\theta_0=W_\theta$.

Fix $\xi<\theta$. If the finite part of $\xi$ is even, $Y_\theta\cap W_\theta^{(\xi)}$ contains the dense rank-$\xi$ layer, so its closure is $W_\theta^{(\xi)}$. The inclusion $W_\theta^{(\xi)}\setminus Y_\theta\subseteq W_\theta^{(\xi+1)}$ holds, and the left-hand side contains the dense rank-$(\xi+1)$ layer, so its closure is $W_\theta^{(\xi+1)}$. If the finite part of $\xi$ is odd, $Y_\theta\cap W_\theta^{(\xi)}\subseteq W_\theta^{(\xi+1)}$ contains the dense rank-$(\xi+1)$ layer, so its closure is $W_\theta^{(\xi+1)}$. After removing $Y_\theta$, the remaining set is contained in $W_\theta^{(\xi+2)}$. If $\xi+1<\theta$, it contains the dense rank-$(\xi+2)$ layer, so its closure is $W_\theta^{(\xi+2)}$. If $\xi+1=\theta$, then $\theta$ has even finite part, so $W_\theta^{(\xi+1)}=\{\omega^\theta\}\subseteq Y_\theta$ and the remaining set is empty, again equal to $W_\theta^{(\xi+2)}$. Therefore, for $\xi<\theta$,
\[
T_\theta(W_\theta^{(\xi)})=
\begin{cases}
W_\theta^{(\xi+1)},&\text{if the finite part of $\xi$ is even},\\
W_\theta^{(\xi+2)},&\text{if the finite part of $\xi$ is odd}.
\end{cases}
\]
For $\theta>0$, the last nonempty derivative is the singleton $\{\omega^\theta\}$. If $\omega^\theta\notin Y_\theta$, the inner closure in the definition of $T_\theta$ is empty. If $\omega^\theta\in Y_\theta$, that inner closure is the singleton itself and becomes empty after subtracting $Y_\theta$. Thus $T_\theta(\{\omega^\theta\})=\varnothing$ in either case. This is the endpoint step $\xi=\theta$. For $\theta=0$, both points belong to $Y_0$, and $F^0_1=\varnothing$.

Define $s_0=0$, let $s_\mu=\mu$ for nonzero limit ordinals $\mu$, and set
\[
s_{\mu+m}=\mu+(2m-1)
\quad(0<m<\omega),
\]
where $\mu$ is zero or a limit ordinal. With $W_\theta^{(\xi)}=\varnothing$ for $\xi>\theta$ and $T_\theta(\varnothing)=\varnothing$, a transfinite induction using the displayed formula for $T_\theta$ together with the endpoint calculation gives $F^\theta_\alpha=W_\theta^{(s_\alpha)}$. At a nonzero limit stage $\nu$, for each $\alpha<\nu$ write $\alpha=\mu+m$ with $\mu$ zero or limit and $m<\omega$. Then $\alpha\leq s_\alpha$, while $s_\alpha<\nu$ because $s_\alpha$ differs from $\mu$ by only a finite ordinal and $\nu$ is limit. Hence $\sup_{\alpha<\nu}s_\alpha=\nu$. Continuity of the Cantor--Bendixson derivatives then gives the limit step. These derivatives decrease strictly while nonempty. Write $\theta=\lambda+n$ with $\lambda$ zero or a limit ordinal and $n<\omega$, and put $\delta=\lambda+\lceil n/2\rceil$. If $n=0$, then $s_\delta=\theta<s_{\delta+1}$. If $n=2r>0$, then $s_\delta=\lambda+2r-1<\theta<s_{\delta+1}=\lambda+2r+1$. If $n=2r+1$, then $s_\delta=\theta<s_{\delta+1}$. Hence in all cases $s_\delta\leq\theta<s_{\delta+1}$, so $F^\theta_\delta\neq\varnothing$ while $F^\theta_{\delta+1}=F^\theta_{\delta+2}=\varnothing$. Since no equality occurs while the derivatives are nonempty, the first stabilization stage is
\[
\gamma=\delta+1=\lambda+\lceil n/2\rceil+1,
\]
and the terminal core is empty. Every nonzero successor ordinal can be written uniquely as $\gamma=\lambda+k$, where $\lambda$ is zero or a limit ordinal and $1\leq k<\omega$. It is obtained by taking $\theta=\lambda+2(k-1)$.

To obtain a nonzero limit stage $\gamma=\lambda$, let $q:W_\lambda\sqcup I\to Q_\lambda$, where $I=[0,1]$, be the quotient map that sends the top point $p=\omega^\lambda$ and $0\in I$ to the same point and is one-to-one elsewhere. The disjoint union $W_\lambda\sqcup I$ is compact Hausdorff. The quotient relation is the diagonal together with the two off-diagonal pairs joining $p$ and $0\in I$, and it is closed. Therefore $Q_\lambda$ is compact Hausdorff. The restrictions of $q$ to the compact summands are continuous injections and hence homeomorphisms onto closed images. Throughout this paragraph, the symbols $W_\lambda$, $I$, and $Y_\lambda$ denote these images in $Q_\lambda$. The common image of the original point $p$ and the endpoint $0$ of the interval is still denoted by $p$, and $\mathbb Q\cap I$ denotes the image of the rational points of the interval. Put
\[
Y_\lambda^*=Y_\lambda\cup(\mathbb Q\cap I).
\]
Since $Y_\lambda$ is dense and $\sigma$-compact in $W_\lambda$ and $\mathbb Q\cap I$ is countable and dense in $I$, their union $Y_\lambda^*$ is dense and $\sigma$-compact in $Q_\lambda$, hence Lindel\"of $\Sigma$. If $U$ is any open neighborhood of $Y_\lambda^*$ in $Q_\lambda$, then $Q_\lambda\setminus U$ is compact and therefore Lindel\"of $\Sigma$. Thus $Y_\lambda^*$ is a Lindel\"of $\Sigma$ kernel of $Q_\lambda$. Let $\widetilde T_\lambda$ be the operator from Definition~\ref{def:core-operator} associated with $(Q_\lambda,Y_\lambda^*)$, and let $(\widetilde F^\lambda_\alpha)$ be its iteration from Theorem~\ref{thm:core-iteration}. Since $Y_\lambda^*\cap I=\mathbb Q\cap I$ and both this set and its complement are dense in $I$, Proposition~\ref{prop:core-fixed-points} gives $\widetilde T_\lambda(I)=I$. Since the top point $p=\omega^\lambda$ has rank $\lambda$ and $\lambda$ is limit, $p\in Y_\lambda$. The summand images meet only at $p$, so $Y_\lambda^*\cap W_\lambda=Y_\lambda$. As $W_\lambda$ is closed in $Q_\lambda$, closures of its subsets are unchanged. Hence $\widetilde T_\lambda$ restricts to $T_\lambda$ on closed subsets of $W_\lambda$. The finite-union clause of Proposition~\ref{prop:core-fixed-points} gives, for closed sets $A,B\subseteq Q_\lambda$,
\[
\widetilde T_\lambda(A\cup B)=\widetilde T_\lambda(A)\cup\widetilde T_\lambda(B).
\]
For any nonempty index set $\Lambda$ and any family $(A_\alpha)_{\alpha\in\Lambda}$ of subsets of $Q_\lambda$,
\[
\bigcap_{\alpha\in\Lambda}(A_\alpha\cup I)=\left(\bigcap_{\alpha\in\Lambda} A_\alpha\right)\cup I.
\]
A transfinite induction therefore gives
\[
\widetilde F^\lambda_\alpha=F^\lambda_\alpha\cup I
\qquad\text{for every ordinal }\alpha.
\]
The base step is $Q_\lambda=W_\lambda\cup I$. The successor step uses the finite-union identity together with $\widetilde T_\lambda(I)=I$ and the restriction to $T_\lambda$, while the nonzero-limit step uses the displayed intersection identity.
For every $\alpha<\lambda$, one has $s_\alpha,s_{\alpha+1}<\lambda$, so the point $p$, whose rank is $\lambda$, belongs to both $F^\lambda_\alpha$ and $F^\lambda_{\alpha+1}$. Hence the strict decrease $F^\lambda_\alpha\supsetneq F^\lambda_{\alpha+1}$ removes a point of $W_\lambda\setminus\{p\}$, and therefore $\widetilde F^\lambda_\alpha\ne\widetilde F^\lambda_{\alpha+1}$. At stage $\lambda$, one has $F^\lambda_\lambda=\{p\}$, so $\widetilde F^\lambda_\lambda=I$, while $F^\lambda_{\lambda+1}=\varnothing$ gives $\widetilde F^\lambda_{\lambda+1}=I$. Hence the first stabilization stage is exactly $\lambda$. Stage zero is realized by $I$ with kernel $\mathbb Q\cap I$.

For $\theta>0$,
\[
|Y_\theta|=|W_\theta|=\max\{\omega,|\theta|\}:
\]
finite Cantor normal forms give the upper bound, while the points $\omega^\xi$ for $\xi<\theta$ and the finite ordinals give the lower bound. The successor ordinals below $\omega^\theta$ are isolated and already have cardinality $|W_\theta|$. Since $\theta>0$, the ordinal $\omega^\theta$ is limit, and $\alpha\mapsto\alpha+1$ injects $\omega^\theta$ into this set. Hence $Y_\theta$, which contains all isolated points, has the same cardinality as $W_\theta$. Thus, for $\gamma<\kappa^+$, the preceding constructions use kernels of size at most $\kappa$. The interval in the limit construction adds only countably many kernel points. If $\gamma\geq1$, add a disjoint clopen compact component $[0,\kappa]$ and include it entirely in the kernel. On the original clopen component, the closed-core operator for the enlarged space restricts to the original operator, since closures of subsets of that component are unchanged and the enlarged kernel has the same trace there. The finite-union clause of Proposition~\ref{prop:core-fixed-points} then shows that the enlarged operator acts componentwise on the disjoint clopen sum. The added component, being contained in the kernel, therefore disappears at the first step. From stage $1$ onward the iterates agree with those of the original construction. Since $\gamma\geq1$, the first stabilization stage remains $\gamma$, and the kernel now has size exactly $\kappa$. If $\gamma=0$, use the original interval $I=[0,1]$ and set
\[
X=I\times[0,\kappa],\qquad
Y=(\mathbb Q\cap I)\times[0,\kappa].
\]
Then $Y$ is $\sigma$-compact of size $\kappa$, and both $Y$ and $X\setminus Y$ are dense in $X$. If $T$ denotes the operator associated with this kernel, Proposition~\ref{prop:core-fixed-points} gives $T(X)=X$, so the first stabilization stage is $0$. In both constructions, $Y$ contains a discrete subspace of size $\kappa$, namely the successor ordinals in the added component when $\gamma\geq1$ and $\{0\}\times\{\alpha+1:\alpha<\kappa\}$ when $\gamma=0$. The singleton cover of this discrete subspace has no subcover of size below $\kappa$, so $\kappa\leq hL(Y)\leq|Y|=\kappa$. In each construction the resulting $Y$ is $\sigma$-compact and hence Lindel\"of $\Sigma$. Since the ambient space is compact, the complement of every open neighborhood of $Y$ is compact. Therefore each resulting $Y$ is a Lindel\"of $\Sigma$ kernel of its ambient space.
\end{proof}

\section{The remaining problem}

By Theorem~\ref{thm:exact-criterion}, the remaining question is whether, for every charming space $X$, every Lindel\"of $\Sigma$ kernel $Y$ of $X$, and every neighborhood assignment $\varphi$ on $X$, there is a set $E\subseteq X$ that is closed discrete in $X$ and satisfies
\[
Y\subseteq\varphi[E].
\]
For a fixed charming space $X$ and neighborhood assignment $\varphi$ on $X$, Theorem~\ref{thm:closed-LSigma-kernel} gives a countable assignment kernel with Lindel\"of $\Sigma$ closure, and Corollary~\ref{cor:two-step} gives a closed discrete set $D$ with
\[
\varphi[\varphi[D]]=X.
\]
The classical $z$-space obstruction, together with Corollary~\ref{cor:no-square-refinement}, Theorem~\ref{thm:closed-lift}, and Corollary~\ref{cor:no-closed-lift}, shows that the extra iteration cannot in general be removed by a universal pointwise square refinement, even after passage through a continuous closed surjection. The closed assignment-kernel construction and this square-refinement obstruction therefore isolate Theorem~\ref{thm:exact-criterion}\textup{(ii)} as the unresolved one-step problem.

On the counterexample side, Theorem~\ref{thm:hfcw-charming}\textup{(iii)--(iv)} shows that every uncountable $\mathrm{HFC}_{\mathrm w}$ subspace of $[\omega_1]^{\leq\omega}$ is dually closed countable but not charming. In particular, Corollary~\ref{cor:tsaban-not-charming} applies this conclusion to Tsaban's version~2 $\diamondsuit$ construction. Thus dually closed countability alone does not decide the one-step problem. Corollary~\ref{cor:terminal-counterexample} gives the complementary reduction on the charming side. Every charming non-$D$ space has a nonempty closed charming terminal core $F_*$ for which both $Y\cap F_*$ and $F_*\setminus Y$ are dense in $F_*$, and this core is itself non-$D$. Consequently, every charming space is a $D$-space if and only if every terminal core arising from the closed-core iteration is a $D$-space.

\Needspace{10\baselineskip}

\end{document}